\documentclass[12pt,reqno]{amsart}
\usepackage{fullpage}
\allowdisplaybreaks
\usepackage{times}
\usepackage{colonequals}
\usepackage{amsmath,amssymb,amsthm,url}
\usepackage[utf8]{inputenc}
\usepackage[english]{babel}
\usepackage{bbm}
\usepackage{enumerate}
\usepackage{bm}
\usepackage{graphicx}
\usepackage{mathrsfs}
\usepackage{makecell}
\usepackage[colorlinks=true, pdfstartview=FitH, linkcolor=blue, citecolor=blue, urlcolor=blue]{hyperref}
\usepackage{caption} 
\newtheorem{thm}{Theorem}[section]
\newtheorem{lem}[thm]{Lemma}
\newtheorem{prop}[thm]{Proposition}
\newtheorem{cor}[thm]{Corollary}

\theoremstyle{definition}
\newtheorem{defn}[thm]{Definition}
\newtheorem{question}[thm]{Question}
\newtheorem{rem}[thm]{Remark}

\newtheorem{ex}[thm]{Example}

\newcommand{\Q}{\mathbb{Q}} 
\newcommand{\R}{\mathbb{R}}

\newcommand{\Z}{\mathbb{Z}}
\newcommand{\F}{\mathbb{F}}

\title{Sperner systems with restricted differences}
\author{Zixiang Xu}
\address{Extremal Combinatorics and Probability Group, Institute for Basic Science, Daejeon, South Korea}
\email{zixiangxu@ibs.re.kr}
\author{Chi Hoi Yip}
\address{School of Mathematics\\ Georgia Institute of Technology\\Atlanta, GA 30332\\ United States}
\email{cyip30@gatech.edu}
\date{}

\begin{document}

\begin{abstract}
Let $\mathcal{F}$ be a family of subsets of $[n]$ and $L$ be a subset of $[n]$. We say $\mathcal{F}$ is an $L$-differencing Sperner system if $|A\setminus B|\in L$ for any distinct $A,B\in\mathcal{F}$. Let $p$ be a prime and $q$ be a power of $p$. Frankl first studied $p$-modular $L$-differencing Sperner systems and showed an upper bound of the form $\sum_{i=0}^{|L|}\binom{n}{i}$. In this paper, we obtain new upper bounds on $q$-modular $L$-differencing Sperner systems using elementary $p$-adic analysis and polynomial method, extending and improving existing results substantially. Moreover, our techniques can be used to derive new upper bounds on subsets of the hypercube with restricted Hamming distances. One highlight of the paper is the first analogue of the celebrated Snevily's theorem in the $q$-modular setting, which results in several new upper bounds on $q$-modular $L$-avoiding $L$-intersecting systems. In particular, we improve a result of Felszeghy, Heged\H{u}s, and R\'{o}nyai, and give a partial answer to a question posed by Babai, Frankl, Kutin, and \v{S}tefankovi\v{c}.
\end{abstract}

\maketitle

\section{Introduction}
Throughout the paper, $n$ is a positive integer and $[n]$ denotes the set $\{1,2,\ldots, n\}$. The set of all subsets of $[n]$ is denoted by $2^{[n]}$, and $\binom{[n]}{k}$ denotes the collection of all subsets of $[n]$ of size $k$. Let $p$ be a prime and $q$ be a power of $p$. Given a positive integer $m$ and a set $L \subseteq \Z$, we write $r \in L \pmod m$ if $r \equiv \ell \pmod m$ for some $\ell \in L$.

A set system $\mathcal{F} \subseteq 2^{[n]}$ is said to be a \emph{Sperner system} (or an \emph{antichain}) if $A \not \subseteq B$ for any pair $A,B$ of distinct sets in $\mathcal{F}$. The celebrated Sperner's theorem \cite{S} states that if $\mathcal{F} \subseteq 2^{[n]}$ is Sperner, then $|\mathcal{F}| \le \binom{n}{\lfloor n/2\rfloor}$. In 1985, Frankl~\cite{F85} first showed the following refined version of Sperner's theorem:
\begin{thm}[Frankl]\label{F85}
Let $p$ be a prime and let $L \subseteq [p-1]$ with $|L|=s$. If $\mathcal{F}\subseteq 2^{[n]}$ such that $|A \setminus B| \in L \pmod p$ for any $A,B \in \mathcal{F}$ such that $A \not \subseteq B$, then 
\begin{equation*}
  |\mathcal{F}| \leq \sum_{i=0}^{s} \binom{n}{i}.
\end{equation*}
\end{thm}

In the same paper, Frankl~\cite{F85} asked whether the above upper bound on $|\mathcal{F}|$ can be improved to $\binom{n}{s}$ provided that $\mathcal{F}$ is a Sperner system. This question has been answered for certain special parameters~\cite{2017MJCNT,2023StabilityComb}, but it remains widely open in general.

Results of similar flavors have been studied extensively in the context of $L$-intersecting systems; we refer to the excellent survey by Frankl and Tokushige~\cite{FT}. Let $L \subseteq \{0,1,\ldots, n\}$. Recall a set system $\mathcal{F} \subseteq 2^{[n]}$ is said to be \emph{$L$-intersecting} if $|A \cap B| \in L$ for any distinct $A,B$ in $\mathcal{F}$, and $\mathcal{F}$ is said to be \emph{$L$-avoiding} if $|A| \not \in L$ for each $A$ in $\mathcal{F}$. A result related to the maximum size of $L$-intersecting systems can be regarded as a refinement of the classical Erd\H{o}s-{K}o-{R}ado theorem~\cite{EKR}. Modular versions of $L$-intersecting systems are also well-studied. Let $m$ be a positive integer and $L \subseteq \{0,1,\ldots, m\}$. We say a set system $\mathcal{F} \subseteq 2^{[n]}$ is \emph{$m$-modular $L$-intersecting} if $|A \cap B|\in L  \pmod m $ for any distinct $A,B$ in $\mathcal{F}$, and \emph{$m$-modular $L$-avoiding} if $|A|\not \in L  \pmod m $ for any $A$ in $\mathcal{F}$. We refer to Section~\ref{sec: background} for a short survey of relevant results, which provides extra background and puts our main results in context.

In the same spirit, we say a set system $\mathcal{F} \subseteq 2^{[n]}$ to be \emph{$L$-differencing Sperner} if $|A \setminus B| \in L$ for any distinct $A,B$ in $\mathcal{F}$, where $L \subseteq [n]$. Note if $0 \notin L$, then an $L$-differencing Sperner set system is indeed Sperner. Let $m$ be a positive integer and $L \subseteq [m-1]$. We say a set system $\mathcal{F} \subseteq 2^{[n]}$ is \emph{$m$-modular $L$-differencing Sperner} if $|A \setminus B| \in L  \pmod m $ for any distinct $A,B$ in $\mathcal{F}$. 

Liu and Liu \cite[Theorem 1.4]{LL09} provided the following refinement on Theorem~\ref{F85}.
\begin{thm}[Liu/Liu]\label{LL}
Let $p$ be a prime and let $L \subseteq [p-1]$ with $|L|=s$. If $\mathcal{F}\subseteq 2^{[n]}$ is $p$-modular $L$-differencing Sperner, then 
\begin{equation*}
  |\mathcal{F}| \leq \sum_{i=0}^{s} \binom{n-1}{i}.
\end{equation*}
\end{thm}

Inspired by \cite{01JCTA}, the following $q$-modular version result was proved by Xu and Liu~\cite{XL12}.

\begin{thm}[Xu/Liu] \label{XL}
Let $L=[s]$ and let $q$ be a prime power such that $q>s$. If $\mathcal{F}\subseteq 2^{[n]}$ is $q$-modular $L$-differencing Sperner, then 
\begin{equation*}
  |\mathcal{F}| \leq \sum_{i=0}^{s} \binom{n}{i}.
\end{equation*}
\end{thm}

Our first result improves and extends Theorem~\ref{XL} substantially. Note that the case that $L$ is an interval is of particular interest; we refer to Section~\ref{sec: interval} for related discussions.

\begin{thm}\label{bchooses}
Let $L=\{b-s+1, b-s+2,\ldots, b\}$ such that $s \leq b<q$, where $q$ is a power of a prime $p$. Assume that $p \nmid \binom{b}{s}$. If $\mathcal{F}\subseteq 2^{[n]}$ is $q$-modular $L$-differencing Sperner, then
\begin{equation*}
  |\mathcal{F}| \leq \sum_{i=0}^{s} \binom{n-1}{i}.
\end{equation*}
\end{thm}

We refer to Remark~\ref{motivation} for some motivations behind Theorem~\ref{bchooses}. To the best knowledge of the authors, Theorem~\ref{bchooses} is the first instance where an analogue of Snevily's theorem (Theorem~\ref{S03}) holds in the $q$-modular setting. Moreover, by taking $L=[q-1]$, Theorem~\ref{bchooses} allows us to deduce the following ``$q$-modular Sperner theorem" immediately:
\begin{cor}\label{cor:qSperner}
Let $q$ be a prime power. Let $\mathcal{F}\subseteq 2^{[n]}$ be a $q$-modular Sperner system, that is, $|A \setminus B| \not \equiv 0 \pmod q$ for any distinct $A,B \in \mathcal{F}$. Then
\begin{equation*}
  |\mathcal{F}| \leq \sum_{i=0}^{q-1} \binom{n-1}{i}.
\end{equation*}
\end{cor}

Note that if $q$ is a prime power, then the above theorem says that we always have the polynomial bound $O(n^{q-1})$ for a $q$-modular Sperner system. In fact, given that $\mathcal{F}=\binom{[n]}{q-1}$ is a $q$-modular Sperner system, the above upper bound is close to sharp. However, this is not true for an $m$-modular Sperner system if $m$ is not a prime power; see Remark~\ref{Non-prime-power moduli}.

Our next theorem shows that a similar result holds (with a slightly more complicated condition) if $L$ is an arithmetic progression. In particular, this allows us to extend Theorem~\ref{XL} to all homogeneous arithmetic progressions; see the deduction in Example~\ref{ex: AP}.  Recall that for a prime $p$ and an integer $n$, $v_p(n)$ denotes the largest non-negative integer $k$ such that $p^k \mid n$.

\begin{thm}\label{thm: AP}
Let $q$ be a power of a prime $p$. Let $L \subseteq [q-1]$ be an arithmetic progression $\{a,a+d,\ldots, a+(s-1)d\}$, where $a$ and $d$ are positive integers. Let $\mathcal{F}\subseteq 2^{[n]}$ be $q$-modular $L$-differencing Sperner. If $\sum_{\ell \in L} v_p(\ell)<\max\{(s-1)v_p(d)+v_p(q), sv_p(d)+ v_p(s!)+1\}$, then
\begin{equation*}
  |\mathcal{F}| \leq \sum_{i=0}^{s} \binom{n}{i}.
\end{equation*}
\end{thm}

Using Theorem~\ref{bchooses}, we can deduce upper bounds on $q$-modular $L$-differencing Sperner systems for an arbitrary interval $L \subseteq [q-1]$. Before stating the theorem, we need to introduce a notation. Let $q=p^k$. For each $1 \leq s \leq q-1$, we can write $s=(s_1, s_2, \ldots, s_k)_p$ in base-$p$ and define 
\begin{equation}\label{mu}
\mu_q(s)=s+\frac{q}{p^j}-p^{v_p(s)},    
\end{equation}
where $j$ is the smallest integer such that $s_j \neq p-1$. If $s=q-1$, then we simply define $\mu_q(s)=s$. Note that we always have $\mu_q(s)<q-1$ unless $s=q-1$.

\begin{thm}\label{thm: closure}
Let $q$ be a power of a prime $p$. Let $L \subseteq [q-1]$ be an interval of size $s$. If $\mathcal{F}\subseteq 2^{[n]}$ is $q$-modular $L$-differencing Sperner, then 
\begin{equation*}
  |\mathcal{F}| \leq \min \bigg\{\sum_{i=0}^{\mu_q(s)} \binom{n-1}{i} \quad, \quad \sum_{i=0}^{2^{s-1}} \binom{n}{i}\bigg\}.
\end{equation*}
Moreover, if $q=p^2$, then the following upper bound also holds:
\begin{equation*}
  |\mathcal{F}| \leq \sum_{i=0}^{2s-1} \binom{n}{i}.
\end{equation*}
\end{thm}

The upper bounds in the above theorems also apply to problems with restricted symmetric differences, which were also widely studied~\cite{D73,HKP20,K66}; we refer to Section~\ref{sec:symmdiff} for a brief discussion. Furthermore, minimal modifications to the proof of the above theorems allow us to deduce new upper bounds on $q$-modular $L$-avoiding $L$-intersecting systems, which improve previous results significantly for certain ranges of $q$ and $|L|$; see Section~\ref{sec:intersecting}. Let $L \subseteq \{0,1,\ldots, q-1\}$ with $|L|=s$, and let $\mathcal{F} \subseteq 2^{[n]}$ be a $q$-modular $L$-avoiding $L$-intersecting system. In the following table, for different $L$ and $q$, we compare our upper bounds on $|\mathcal{F}|$ with the best-known upper bounds.

\begin{table}[ht]
\centering
\begin{tabular}{|m{4cm}|m{6cm}|m{4.5cm}|}
\hline
Requirements on $q$ and $L$ & New upper bounds & Best-known upper bounds  \\ \hline
general $L$ & $\sum\limits_{i=0}^{q-1}\binom{n}{i}$  [Theorem~\ref{improve_2^sss}] & $\sum\limits_{i=0}^{2^{s-1}}\binom{n}{i}$ [Theorem~\ref{2^sss}] \\ \hline
$L$ is an interval (in the modulo $q$ sense) & $\sum\limits_{i=0}^{\mu_q(s)}\binom{n}{i}$  [Theorem~\ref{improve_FHR}]& $\sum\limits_{i=s}^{q-1}\binom{n}{i}$ [Theorem~\ref{FHR}]\\ \hline
$L$ is an interval, $q=p^2$ & $\sum\limits_{i=0}^{2s-1}\binom{n}{i}$  [Theorem~\ref{p^2}] & $\sum\limits_{i=0}^{s^2/4+1}\binom{n}{i}$ [Theorem~\ref{2^sss}]\\ \hline
\end{tabular}
\caption{Comparisons between our new upper bounds and the best-known upper bounds on $q$-modular $L$-avoiding $L$-intersecting systems}
\label{table:comparisons}
\end{table}

By taking a prime $p>n$, Theorem~\ref{LL} implies the upper bound $\sum_{i=0}^s \binom{n-1}{i}$ on an $L$-differencing Sperner system from $2^{[n]}$. Under extra assumptions on the size of sets in $\mathcal{F}$, this upper bound can be improved \cite[Theorem 1.4]{LZ16}. However, it is more interesting to explore if this upper bound can be improved without any additional assumption. We have mentioned that the case $L=[s]$ is of special interest, especially because the lower bound $\binom{n}{s}$ is readily available and often believed to be sharp. Indeed, as an immediate corollary of the main result in~\cite{2017MJCNT}, Frankl showed the lower bound $\binom{n}{s}$ is sharp when $s=O(\sqrt{n})$ and asked if one can go beyond $O(\sqrt{n})$ \cite[Section 7]{2017MJCNT}. While we are not able to achieve this, we show an improved upper bound in the following theorem when $n/3<s\leq n/2$.

\begin{thm}\label{[s] sym}
Let $L=[s]$ such that $(n+2)/3\leq s \leq n/2$. Let $\mathcal{F} \subseteq 2^{[n]}$ such that $\mathcal{F}\subseteq 2^{[n]}$ is $L$-differencing Sperner. Then
$$
|\mathcal{F}|\leq \sum_{i=3s-n-1}^s \binom{n-1}{i}.
$$
\end{thm}

Recently, Nagy and Patk\'{o}s~\cite{2021GC} introduced the notion of $L$-close Sperner systems. For a set $L$ of positive integers, a set system $\mathcal{F} \subseteq 2^{[n]}$ is said to be {\em $L$-close Sperner}, if for any pair of distinct sets $A,B$ in $\mathcal{F}$, the skew distance $sd(A,B)=\min\{|A \setminus B|,|B \setminus A|\} \in L$. Boros, Gurvich, and Milani\v c \cite{2019EJC,2020JGT} also introduced similar notions and their motivations are from computer science. Note that an $L$-differencing Sperner system is an $L$-close Sperner system, but not vice versa. Nagy and Patk\'{o}s~\cite{2021GC} proved the following upper bound on $L$-close Sperner systems:

\begin{thm}[Nagy/Patk\'{o}s]\label{NPmainthm}
Let $L$ be a set of $s$ positive integers. If $\mathcal{F}\subseteq 2^{[n]}$ is $L$-close Sperner, then we have
\begin{equation*}
  |\mathcal{F}| \leq \sum_{i=0}^{s} \binom{n}{i}.
\end{equation*}
\end{thm}
Moreover, when $|L|=1$, they showed that $|\mathcal{F}| \leq n$. They conjectured that if $L=[s]$ and $\mathcal{F}\subseteq 2^{[n]}$ is $L$-close Sperner, then $|\mathcal{F}|\leq \binom{n}{s}$, which is sharp by considering $\binom{[n]}{s}$. We make partial progress and prove the following theorem:

\begin{thm}\label{thm: [s]}
Let $L=[s]$ such that $(n+1)/3 \leq s \leq n/2$. Let $\mathcal{F} \subseteq 2^{[n]}$ be $L$-close Sperner. Then 
$$
|\mathcal{F}|\leq \sum_{i=3s-n}^s \binom{n}{i}.
$$
\end{thm}
In particular, when $n$ is even and $s=n/2$, the above theorem is sharp by Sperner's theorem.

\subsection*{Structure of the paper.}
In Section~\ref{sec: background}, we provide additional background and put our main theorems in context. In Section~\ref{sec:preliminary}, we introduce some useful tools and prove some preliminary results. In Section~\ref{sec: q-modular}, we prove Theorem~\ref{bchooses}, Theorem~\ref{thm: AP}, and Theorem~\ref{thm: closure}. In Section~\ref{sec: push}, we prove Theorem~\ref{[s] sym} and Theorem~\ref{thm: [s]}. Finally, in Section~\ref{sec: Hamming}, we apply our main results to deduce new bounds on intersecting systems and explain how our results extend to the setting of prescribed Hamming distances.

\section{Background and overview of the paper}\label{sec: background}
In this section, we survey some important results in the study of $L$-intersecting systems and compare these results with our main results. In particular, we will review the techniques used in the seminal paper \cite{01JCTA} for $q$-modular $L$-avoiding $L$-intersecting systems and state their analogues in the setting of $q$-modular $L$-differencing Sperner systems (to be proved in later sections).

\subsection{Non-modular and modular versions}

$L$-intersecting systems were first studied by Ray-Chaudhuri and Wilson~\cite{RCW75}. One particular celebrated result in this setting that resembles Theorem~\ref{LL} is the following theorem, due to Snevily \cite{03Combinatorica}.

\begin{thm} [Snevily] \label{S03}
Let $L$ be a set of $s$ positive integers. If $\mathcal{F}\subseteq 2^{[n]}$ is an $L$-intersecting system, then
\begin{equation}\label{form}
  |\mathcal{F}| \leq \sum_{i=0}^{s} \binom{n-1}{i}.
\end{equation}
\end{thm} 

\begin{rem}\label{motivation}
This is the best-known upper bound on $L$-intersecting systems (without additional assumptions on $\mathcal{F}$). Lots of efforts have been made to achieve an upper bound of the same form as \eqref{form} in different variants of extremal set problems; see for example \cite{2007GC,LL09,XLZ18}. In particular, for a $p$-modular $L$-avoiding $L$-intersecting system $\mathcal{F}$ with $|L|=s$, the upper bound \eqref{form} holds (see for example \cite[Theorem 5]{2007GC}). This serves as our main motivation for improving Theorem~\ref{XL} to Theorem~\ref{bchooses}: Theorem~\ref{bchooses} seems to be the first instance where an upper bound of the same form as~\eqref{form} appears in the $q$-modular setting. Other new results in the paper are of a similar flavor. 
\end{rem}

The $p$-modular (and the $q$-modular) $L$-intersecting systems were first studied by Frankl and Wilson \cite{FW81}. We refer to the survey \cite{LY14} by Liu and Yang for related results. The modular version (both for $L$-intersecting systems and $L$-differencing Sperner systems) is interesting and useful if $L$ has some special arithmetic properties (for example, $L$ is contained in the union of a few arithmetic progressions with the same modulus), in which case the upper bound given by the modular version tends to improve the upper bound given in the non-modular version significantly. We illustrate this philosophy in the following example by comparing Theorem~\ref{bchooses} with Theorem~\ref{LL}.

\begin{ex}
Let $n$ be sufficiently large. Let $L$ be the set of primes up to $n$ (together with $1$). Let $\mathcal{F}\subseteq 2^{[n]}$ be an $L$-differencing Sperner system. Note that for any prime $p<n$, Theorem~\ref{LL} does not apply since $p \in L$. Thus, Theorem~\ref{LL} only gives an upper bound on $|\mathcal{F}|$ of order $\binom{n}{\pi(n)+1}$, where $\pi(n)=(1+o(1))\frac{n}{\log n}$ by the prime number theorem. However, taking $L=\{1,2,3\} \pmod 4$, then Theorem~\ref{bchooses} gives 
$$
|\mathcal{F}| \leq \binom{n-1}{3}+\binom{n-1}{2}+\binom{n-1}{1}+\binom{n-1}{0}=\binom{n}{3}+n,
$$
which improves the upper bound given by Theorem~\ref{LL} exponentially. It is interesting to see if the trivial lower bound $\binom{n}{3}$ (given by the construction $\binom{[n]}{3}$) can be improved.
\end{ex}

The $q$-modular version is also useful for various applications. For example, Frankl and Wilson \cite{FW81} derived upper bounds on uniform $q$-modular $L$-intersecting systems and obtained improved lower bounds on the chromatic number of the unit distance graph in $\R^n$ as well as the constructive lower bound for the Ramsey problem. 

\subsection{$p$-adically separating polynomials and $q$-modular $L$-intersecting systems}
For $p$-modular or non-modular results, only the linear algebra methods are required. The only difference is the underlying field used: we work over the field $\F_p$ for the $p$-modular version, while we work over the field $\Q$ for the non-modular version. If $q$ is a prime power, an appropriate underlying field is not available and thus extra efforts are required to obtain $q$-modular results. In particular, one important contribution, due to Babai, Frankl, Kutin, and \v{S}tefankovi\v{c} \cite{01JCTA}, is to convert the problem to finding upper bounds on the degree of $p$-adically separating polynomials. We survey their main results and techniques in this section.

We follow the definitions in \cite[Section 2]{01JCTA} for separating polynomials:
\begin{itemize}
    \item Given a set $L \subseteq \Z$ and an element $\alpha \not \in L$, we say that a univariate polynomial $g \in \Z[y]$ ($p$-adically) \emph{separates $\alpha$ from $L$} if $v_p(g(\alpha))<v_p(g(\ell))$ for each $\ell \in L$.
    \item Let $D(L, \alpha, q)$ denote the minimum possible degree of a polynomial separating $\alpha$ from $L+q\Z$.
    \item Let $D(s, k)$ be the maximum value of $D\left(L, \alpha, p^k\right)$, taken over all primes $p$, all $L \subseteq \{0,1,\ldots, p^k-1\}$ of size $|L|=s$, and all $\alpha \notin L \pmod {p^k}$.
\end{itemize}

The following lemma can be proved by combining the linear algebra methods and a simple $p$-adic argument.

\begin{lem}[{\cite[Lemma 3.1]{01JCTA}}]\label{01}
Let $L \subseteq \{0,1,\ldots, q-1\}$. Assume that for each $\alpha \not \in L \pmod q$, there exists a degree-$d$ univariate polynomial $g_{\alpha}$ separating $\alpha$ from $L+q\Z$. If $\mathcal{F} \subseteq 2^{[n]}$ is a $q$-modular $L$-avoiding $L$-intersecting system, then  
$$|\mathcal{F}|\leq \sum_{i=0}^{d}\binom{n}{i}.$$
\end{lem}

We will prove the following proposition in Section~\ref{4.1}. It can be regarded as a refinement of Lemma~\ref{01} in our Sperner setting. In particular, some new ingredients and extra efforts are required to deduce the stronger upper bound~\eqref{n-1i}.

\begin{prop}\label{prop: p-adic}
Let $q$ be a power of a prime $p$ and let $L \subseteq [q-1]$. Let $\mathcal{F}\subseteq 2^{[n]}$ be a $q$-modular $L$-differencing Sperner system. Assume that there exists a degree-$d$ univariate polynomial $g$ separating $0$ from $L+q\Z$, that is, $v_p(g(0))<v_p(g(u))$ for each $u \in L+q\Z$, then we have \begin{equation*}
  |\mathcal{F}| \leq \sum_{i=0}^{d} \binom{n}{i}.
\end{equation*}
If in addition $v_p(g(0))\leq v_p(g(u-1))$ for each $u \in L+q\Z$, or $v_p(g(0))\leq v_p(g(u+1))$ for each $u \in L+q\Z$, then the following improved upper bound holds:
\begin{equation}\label{n-1i}
  |\mathcal{F}| \leq \sum_{i=0}^{d} \binom{n-1}{i}.
\end{equation}
\end{prop}

In view of the above two results, we are led to study the upper bounds on the degree of separating polynomials. In \cite[Lemma 5.1]{01JCTA}, Babai, Frankl, Kutin, and \v{S}tefankovi\v{c} proved that
\begin{equation}\label{D(s,k)}
D(s,k) \leq \min \bigg\{2^{s-1}, \bigg(1+\frac{s-1}{k}\bigg)^k \bigg\}.    
\end{equation}
For different ranges of $s$ and $k$, the upper bound on $D(s,k)$ can be improved; see \cite[Section 7]{01JCTA} and \cite{KZ08}.
Combing Lemma~\ref{01} and inequality~\eqref{D(s,k)}, they concluded the following \cite[Theorem 1.2]{01JCTA}:

\begin{thm}[Babai/Frankl/Kutin/\v{S}tefankovi\v{c}]\label{2^sss}
Let $q=p^k$ and let $L \subseteq \{0,1,\ldots, q-1\}$ of size $s$. Let $\mathcal{F} \subseteq 2^{[n]}$ be a $q$-modular $L$-avoiding $L$-intersecting system of sets. Then
\begin{equation*}
  |\mathcal{F}| \leq \sum_{i=0}^{D(s,k)} \binom{n}{i} \leq \sum_{i=0}^{2^{s-1}} \binom{n}{i}.
\end{equation*}
\end{thm}

\begin{rem}\label{smallq}
They even showed that the upper bound $2^{s-1}$ for $D(s,k)$ is sharp in \cite[Theorem 6.3]{01JCTA}; however, the proof relies on converting the estimation of $D(s,k)$ to an equivalent optimization problem \cite[Theorem 6.1]{01JCTA}, where it is implicitly assumed that $p>2^s$ \cite[Lemma 6.3]{01JCTA}. Thus, if $q<2^s$, it is likely that Theorem~\ref{2^sss} can be improved and it makes perfect sense if the upper bound can be improved from $O(n^{2^{s-1}})$ to $O(n^{C(q,s)})$ for some polynomial $C$ depending on both $q$ and $s$; we confirm this in Section~\ref{sec:intersecting}.
\end{rem}

Next, we show that the same upper bound holds for all $q$-modular $L$-differencing Sperner systems.

\begin{thm}\label{2^s}
Let $q$ be a prime power and let $L \subseteq [q-1]$ of size $s$. Let $\mathcal{F} \subseteq 2^{[n]}$ be a $q$-modular $L$-differencing Sperner system. Then
\begin{equation*}
  |\mathcal{F}| \leq \sum_{i=0}^{2^{s-1}} \binom{n}{i}.
\end{equation*}
\end{thm}

\begin{proof}
It follows from the first part of Proposition~\ref{prop: p-adic} and the upper bound~\eqref{D(s,k)} on the degree of $p$-adically polynomials that separates $0$ and $L+q\Z$.
\end{proof}

We finish the section by describing a very different situation for the $m$-modular version, where $m$ is not a prime power.

\begin{rem}\label{Non-prime-power moduli}
In view of Theorem~\ref{2^sss}, we have a polynomial upper bound, that is, of the form $O(n^{c(s)})$ for some function $c(s)$, for $q$-modular $L$-avoiding $L$-intersecting systems with $|L|=s$ over $2^{[n]}$, whenever $q$ is a prime power. We also see a similar phenomenon in the setting of $L$-differencing Sperner systems in Theorem~\ref{2^s}.

However, both statements fail to extend to the $m$-modular version, where $m$ is not a prime power. Indeed, Grolmusz \cite{G00} showed that for each $m$ with at least $2$ distinct prime divisors, there is an $m$-modular $[m-1]$-avoiding $[m-1]$-intersecting system $\mathcal{F} \subseteq 2^{[n]}$ with super-polynomial size; note that $\mathcal{F}$ is also an $m$-modular $[m-1]$-differencing Sperner system since each set $A \in \mathcal{F}$ satisfies $|A| \equiv 0 \pmod m$. We refer to Kutin~\cite{K02} for a related discussion.
\end{rem}

\subsection{Improved upper bounds for intervals}\label{sec: interval}
In \cite[Section 10]{01JCTA}, the authors suspected that the upper bound given in Theorem~\ref{2^sss} is far away from the truth in general. Indeed, when $L$ is an interval of the form $\{0,1, \ldots, s-1\}$, they showed the following improvement.

\begin{thm}[{\cite[Corollary 9.1]{01JCTA}}]\label{2s}
Let $q$ be a prime power and let $L=\{0,1, \ldots, s-1\}$ with $s<q$. Let $\mathcal{F} \subseteq 2^{[n]}$ be a $q$-modular $L$-avoiding $L$-intersecting system of sets. Then 
\begin{equation*}
  |\mathcal{F}| \leq \sum_{i=0}^{2s} \binom{n}{i}.
\end{equation*}
\end{thm}

Using a combination of Gr\"{o}bner basis methods and linear algebra, Heged\H{u}s and R\'{o}nyai \cite{HR03} proved the following theorem, extending a classical result by Frankl and Wilson \cite{FW81}.

\begin{thm}[Heged\H{u}s/R\'{o}nyai]\label{HR}
Let $\mathcal{F} \subseteq 2^{[n]}$ such that $|A| \equiv k \pmod q$ for each $A \in \mathcal{F}$, and $|A \cap B| \not \equiv k \pmod q$ for each $A \neq B \in \mathcal{F}$. If $2(q-1) \leq n$, then $|\mathcal{F}|\leq \binom{n}{q-1}$.
\end{thm}

Moreover, with extra work, Felszeghy, Heged\H{u}s, and R\'{o}nyai \cite[Theorem 1.3] {FHR09} extended Theorem~\ref{HR} to all intervals $L$:

\begin{thm} [Felszeghy/Heged\H{u}s/R\'{o}nyai]\label{FHR}
Let $L \subseteq \{0,1,\ldots, q-1\}$ be an interval (in the modulo $q$ sense) and let $\mathcal{F} \subseteq 2^{[n]}$ be a $q$-modular $L$-avoiding $L$-intersecting system. If $|L| \leq n-q+2$, then 
$$|\mathcal{F}|\leq \sum_{i=|L|}^{q-1}\binom{n}{i}.$$
\end{thm}

The above results are all consistent with the predictions in Remark~\ref{smallq}. We will provide an improvement on Theorem~\ref{FHR} in Theorem~\ref{improve_FHR}.

The proof of the upper bound~\eqref{D(s,k)} on $D(s,k)$ relies on the observation that one can create a separating polynomial based on the leaves in the ``closure" of the trie \footnote{A trie over a finite alphabet is a rooted tree whose edges are labeled
by elements of the alphabet.} (over the alphabet $\{0,1,\ldots, p-1\}$) associated to $L$ \cite[Section 4 and Section 5]{01JCTA}. Motivated by this observation, in our Sperner setting, we introduce the following definitions of ``closure" of a set $L$ in view of Theorem~\ref{bchooses}.

\begin{defn}
Let $q$ be a power of a prime $p$. Let $L=\{b-s+1,\ldots, b\} \subseteq [q-1]$ be an interval. We say $L$ is \emph{$q$-closed} if $p \nmid \binom{b}{s}$.
\end{defn}

\begin{defn}
Let $q$ be a power of a prime $p$. Let $L \subseteq [q-1]$. A \emph{$q$-closure} of $L$ is a shortest interval $L' \subseteq [q-1]$ such that $L \subseteq L'$ and $L'$ is $q$-closed.
\end{defn}

The $q$-closure of a set $L \subseteq [q-1]$ may not be unique. For example, if $L=\{2, \ldots, p\}$, then $L$ is not $q$-closed since $p \mid \binom{p}{2}$, while $L'=[p]$ and $L''=\{2, \ldots, p+1\}$ are both $q$-closures of $L$.

Given these definitions, Theorem~\ref{bchooses} provides a nice upper bound on $\mathcal{F}$ provided that $L$ is $q$-closed. In general, we may first take a $q$-closure of $L$ and then apply Theorem~\ref{bchooses}. The following lemma, to be proved in Section~\ref{padic}, would be useful in proving Theorem~\ref{thm: closure}.

\begin{lem}\label{lem: closure}
Let $q$ be a power of a prime $p$. Let $L\subseteq [q-1]$ be an interval of size $s$ and let $L'$ be a $q$-closure of $L$. Then $|L'| \leq \mu_q(s)$, where $\mu_q(s)$ is defined in equation~\eqref{mu}. 
\end{lem}

From the lemma, we can see that a $q$-closure of $L$ tends to be much smaller than the ``closure" of the trie associated with $L$ (which has size at most $2^{s-1}$). This observation allows us to obtain improved upper bounds on intersecting systems; see Section~\ref{sec:intersecting}.

\section{Preliminaries}\label{sec:preliminary}

\subsection{Multilinear polynomials}
We will use the linear algebra method to prove our main results. One standard technique in extremal set theory is to replace each polynomial with its multilinear reduction so that the dimension of the space they are living in would become smaller.

Throughout the paper, $x=(x_1, x_2, \ldots, x_n) \in \Q^n$. The multilinear reduction of a monomial $\prod_{i \in I} x_i^{\ell_i}\left(\ell_i \geqslant 1\right)$ is the monomial $\prod_{i \in I} x_i$. The \emph{multilinear reduction} of a polynomial $f$ is obtained by expanding $f$ as a linear combination of monomials and performing the multilinear reduction of each monomial. A simple fact that is useful in our discussion is the following: if $g$ is the multilinear reduction of a polynomial $f$, then $f(x)=g(x)$ whenever $x$ is a $\{0,1\}$-vector.

\subsection{Push to the middle}
Let $G=(V,E)$ be a simple graph, and let $S$ be a subset of $E$. If no two
edges in $S$ are incident, then we say that $S$ is a matching of $G$. Recall a \emph{perfect matching} in a bipartite graph $G=A\cup B$ is an injective mapping $f: A\rightarrow B$ such that for every $x\in A$, there is an edge $e\in E$ with endpoints $x$ and $f(x)$. For a subset $T$ of $V$, let $N_{G}(T)$ denote the set of neighbors of $T$ in $G$. The famous Hall's marriage theorem can be stated as follows.

\begin{thm}\label{thm:Hall}
    For a bipartite graph $G$ on the parts $A$ and $B$, there exists a perfect matching $f: A\rightarrow B$ if and only if for every subset $T\subseteq A$, $|T|\leq|N_{G}(T)|$.
\end{thm}

The following lemma is well-known and can be used to prove Sperner's theorem~\cite{1975Brown}.

\begin{lem}\label{push}
Let $\mathcal{F} \subseteq 2^{[n]}$ be a Sperner system. If the smallest size of sets in $\mathcal{F}$ is $k$ with $2k \leq n$, then there is an injective function $f:\mathcal{F} \to 2^{[n]}$ such that 
\begin{itemize}
    \item $f(A)=A$ for each $A \in \mathcal{F}$ with $|A|> k$.
    \item $f(A)=A \cup \{x\}$ for some $x \notin A$ for each $A \in \mathcal{F}$ with $|A|=k$.
    \item $f(\mathcal{F})$ is a Sperner system.
\end{itemize}  
\end{lem}

Next, we use Lemma~\ref{push} to deduce the following corollary.

\begin{cor}\label{cor: middle}
Let $L=[s]$ such that $2s \leq n$. Let $\mathcal{F} \subseteq 2^{[n]}$ be $L$-close ($L$-differencing, resp.) Sperner. Then there is $\mathcal{F}' \subseteq 2^{[n]}$ such that
\begin{itemize}
    \item $|\mathcal{F}'|=|\mathcal{F}|$
    \item $s \leq |A| \leq n-s$ for each $A \in \mathcal{F}'$, 
    \item $\mathcal{F}'$ is $L$-close ($L$-differencing, resp.) Sperner.
\end{itemize}
\end{cor}
\begin{proof}
By applying Lemma~\ref{push} inductively, we can find an injective function $f:\mathcal{F} \to 2^{[n]}$ such that 
\begin{itemize}
    \item $f(A)=A$ for each $A \in \mathcal{F}$ with $|A|\geq s$.
    \item $|f(A)|=s$ and $f(A) \supset A$, for each $A \in \mathcal{F}$ with $|A|<s$.
    \item $f(\mathcal{F})$ is Sperner.
\end{itemize}
By a similar argument, we can also replace each set $A \in \mathcal{F}$ such that $|A|>n-s$ with a subset of $A$ of size $n-s$. Thus, we can find an injective function $g:\mathcal{F} \to 2^{[n]}$ such that 
\begin{itemize}
    \item $g(A)=A$ for each $A \in \mathcal{F}$ with $s \leq |A|\leq n-s$.
    \item $|g(A)|=s$ and $g(A) \supset A$, for each $A \in \mathcal{F}$ with $|A|<s$.
    \item $|g(A)|=n-s$ and $g(A) \subseteq A$, for each $A \in \mathcal{F}$ with $|A|>n-s$.
    \item $\mathcal{F}':=g(\mathcal{F})$ is Sperner.
\end{itemize}
It remains to show $\mathcal{F}'$ is $L$-close Sperner or $L$-differencing Sperner.

Let $A,B \in \mathcal{F}$ such that $|A \setminus B|\leq s$. Next we show that $|g(A) \setminus g(B)| \leq s$:
\begin{itemize}
    \item If $|A|<s$, then $|g(A) \setminus g(B)|\leq |g(A)|=s$.
    \item If $|B|>n-s$, then $|g(A) \setminus g(B)|=|g(A) \cap ([n] \setminus g(B))| \leq n-|g(B)|=s$.
    \item If $s \leq |A| \leq n-s$ and $|B| \leq n-s$, then $g(A)=A$ and $g(B) \supset B$. Thus, $|g(A) \setminus g(B)|\leq |A \setminus B|\leq s$.
        \item If $|A|>n-s$ and $|B|\leq n-s$, then $g(A) \subseteq A$ and $g(B) \supset B$, thus $|g(A) \setminus g(B)|\leq |A \setminus B|\leq s$.
\end{itemize}

We conclude that if $\mathcal{F}$ is $L$-close Sperner, then $\mathcal{F}'$ is also $L$-close Sperner; and if $\mathcal{F}$ is $L$-differencing Sperner, then $\mathcal{F}'$ is also $L$-differencing Sperner.
\end{proof}

\subsection{$p$-adic valuation}\label{padic}

Let $p$ be a prime. We recall some basic properties of the $p$-adic valuation. For each integer $n$, we define $v_p(n)$ to be the largest non-negative integer $k$ such that $p^k \mid n$. Note that $v_p(0)=+\infty$ and $v_p(ab)=v_p(a)+v_p(b)$ for any integers $a,b$. A basic fact (sometimes known as \emph{the ultrametric inequality}) that is useful for our discussions is the following: if $x=y_1+y_2+\cdots+y_m$, then
$$v_p(x) \geq \min \{v_p(y_i): 1 \leq i \leq m\}.$$
Moreover, $v_p(a+b)=\min \{v_p(a), v_p(b)\}$ if $v_p(a) \neq v_p(b)$.

The following fact from elementary number theory will be useful.

\begin{lem}\label{vps!}
Let $q$ be a power of a prime $p$. 
If $1 \leq s<q$, then $v_p(s!) \leq v_p(k(k-1)\cdots (k-s+1))$ for any integer $k$. If in addition there is $0 \leq i \leq s-1$ such that $q|(k-i)$, then the inequality is strict.
\end{lem}
\begin{proof}
If $k-s+1 \leq 0 \leq k$, then we are done. So we can assume $k \geq s$. By Legendre's formula, we have
\begin{align*}
&v_p(k(k-1)\cdots (k-s+1))-v_p(s!)
=v_p(k!)-v_p((k-s)!)-v_p(s!)\\
&=\sum_{j=1}^{\infty} \bigg(\lfloor k/p^j \rfloor -\lfloor (k-s)/p^j \rfloor- \lfloor s/p^j \rfloor\bigg)
\geq \lfloor k/q \rfloor -\lfloor (k-s)/q \rfloor- \lfloor s/q \rfloor,
\end{align*}
where we used the fact that $\lfloor x \rfloor +\lfloor y \rfloor \leq \lfloor x+y \rfloor$ holds for all real numbers $x$ and $y$. Thus, $v_p(k(k-1)\cdots (k-s+1))\geq v_p(s!)$. If there is $0 \leq i \leq s-1$ such that $q|(k-i)$, then $\lfloor k/q \rfloor -\lfloor (k-s)/q \rfloor- \lfloor s/q \rfloor=1$ and thus $v_p(k(k-1)\cdots (k-s+1))> v_p(s!)$.
\end{proof}

We will need to compute the $p$-adic valuation of binomial coefficients. To do that, we recall a classical theorem of Kummer: let $a,b$ be non-negative integers, then $v_p(\binom{a+b}{a})$ is equal to the number of carries when $a$ is added to $b$ in base $p$. The following corollary is an immediate consequence of Kummer's theorem (alternatively one can use Lucas's theorem to derive the same fact):

\begin{cor}\label{nocarries}
If $x=(x_1,x_2,\ldots, x_k)_p$ and $y=(y_1,y_2,\ldots ,y_k)_p$ are written in their base-$p$ representations so that $0 \leq x_i,y_i \leq p-1$, then $p \nmid \binom{x}{y}$ if and only if $y_i \leq x_i$ for all $i$.
\end{cor}

Next, we apply this criterion to deduce Lemma~\ref{lem: closure}:

\begin{proof}[Proof of Lemma~\ref{lem: closure}] 
Let $q=p^k$ and let $L=\{a+1,\ldots, b\} \subseteq [q-1]$ with $s=b-a$. If $s=q-1$, then we must have $L=[q-1]$, which is already $q$-closed. Next we assume that $s<q-1$.
We write $s,a,b$ in base-$p$: 
$$
s=(s_1,s_2, \ldots, s_k)_p, a=(a_1,a_2, \ldots, a_k)_p, b=(b_1,b_2, \ldots, b_k)_p.
$$
Let $j$ be the smallest integer such that $s_j \neq p-1$. In other words, $s_1=s_2=\cdots=s_{j-1}=p-1$ and $s_j<p-1$. This forces $a_1=a_2=\cdots=a_{j-1}=0$,  $b_1=b_2=\cdots=b_{j-1}=p-1$, and $a_j \leq b_j$. Also note that for $k-v_p(s)+1 \leq i \leq k$, we have $s_i=0$ and thus $a_i=b_i$. 
Let 
$$
b'=(b_1', b_2', \ldots, b_k')_p,
$$
where $b_i'=\max\{a_i,b_i\}$ for each $1 \leq i \leq k$. Then it is clear that we have $p \nmid \binom{b'}{a}$ by Corollary~\ref{nocarries}. It follows that the interval $\{a+1, \ldots, b'\}$ is $q$-closed and has size
$$
b'-a=(b'-b)+s \leq s+\sum_{i=j+1}^{k-v_p(s)} (b_j'-b_j)p^{k-i} \leq s+\sum_{i=j+1}^{k-v_p(s)} (p-1)p^{k-i}=s+p^{k-j}-p^{v_p(s)}=\mu_q(s)
$$
since $b_i'=b_i$ for $1 \leq i \leq j$ and $k-v_p(s)+1 \leq i \leq k$. Thus, a $q$-closure of $L$ has size at most $\mu_q(s)$.
\end{proof}

\begin{rem}
It is easy to see that Lemma~\ref{lem: closure} is optimal. For example, if $q=p^2$ and $L=\{p\}$, then it is easy to verify that a $q$-closure of $L$ has size $p=\mu_{p^2}(1)$ from Corollary~\ref{nocarries}.

Given an interval $L \subseteq [q-1]$, it is easy to design an algorithm to find a $q$-closure of $L$ based on Corollary~\ref{nocarries}.
\end{rem}

\section{$q$-modular $L$-differencing Sperner systems}\label{sec: q-modular}

In this section, we derive upper bounds on $q$-modular $L$-differencing Sperner systems.

\subsection{Proof of Proposition~\ref{prop: p-adic}}\label{4.1}

We begin the section with the proof of Proposition~\ref{prop: p-adic}, which allows us to use a separating polynomial to upper bound the size of a $q$-modular $L$-Sperner system. The proof is inspired by the ideas in \cite{2007GC, 03Combinatorica, XL12}. Although some of the technical steps have appeared in previous works in the $p$-modular or the non-modular setting, we include a self-contained proof due to the additional delicate $p$-adic reasoning in the $q$-modular setting.

\begin{proof}[Proof of Proposition~\ref{prop: p-adic}]
Let $g$ be a degree-$d$ univariate polynomial that separates $0$ from $L+q\Z$. Let $\mathcal{F}=\{A_1,A_2,\ldots, A_m\}$. By relabelling the sets in $\mathcal{F}$, we may assume that $n \in A_i$ whenever $i>r$ and $n \notin A_i$ whenever $i \leq r$. For each $1 \leq i \leq m$, let $v^{(i)}$ be the characteristic vector of $A_i$ and define the polynomial 
\begin{equation}\label{gi}
g_{i}(x)=g(|A_{i}|-v^{(i)}\cdot x),    
\end{equation}
where $x=(x_1,x_2,\ldots, x_n)$. For each $1 \leq i \leq m$, let $p_i$ be the multilinear reduction of $g_i$. Note that for each $1 \leq i, j \leq m$, we have $p_i(v^{(j)})=g(|A_i \setminus A_j|)$.

We claim that $\{p_i\}_{i=1}^{m}$ are linearly independent over $\Q$. Suppose $\sum_{i=1}^m \alpha_i p_i=0$ with $\alpha_i$ not all zero. By scaling, we may assume that all $\alpha_i$ are integers and not all $\alpha_i$ are divisible by $p$. Suppose that $\alpha_j$ is not divisible by $p$; then by setting $x=v^{(j)}$, we get
$$
\alpha_j g(0)=\alpha_j p_j(v^{(j)})=-\sum_{i \neq j} \alpha_ip_i(v^{(j)})=-\sum_{i \neq j} \alpha_i g(|A_i \setminus A_j|).
$$
Note that whenever $i \neq j$, we have $|A_i \setminus A_j| \in L \pmod q$ and thus $v_p(g(|A_i \setminus A_j|)) >v_p(g(0))$ by the assumption.  Therefore, by the ultrametric inequality, $$v_p(\alpha_j g(0))\geq \min_{i \neq j} v_p(\alpha_i g(|A_i \setminus A_j|))>v_p(g(0)).$$ 
This implies that  $v_p(\alpha_j) \geq 1$, that is, $\alpha_j$ is divisible by $p$, a contradiction. This proves the claim. 

Note that $p_i$ has degree at most $d$ for each $1 \leq i \leq m$. Thus, $p_1,p_2, \ldots, p_m$ lie in the space of multilinear polynomials with degree at most $d$ in $n$ variable. By counting the dimension of the space, we conclude that $|\mathcal{F}|=m \leq \sum_{i=0}^d \binom{n}{i}$. This proves the first part of the proposition.

Next we assume in addition that $v_p(g(0))\leq v_p(g(u-1))$ for each $u \in L+q\Z$. For the other case, the proof is similar, and we shall explain how to modify the proof in the end.

Label the sets in $$\binom{[n-1]}{0}\sqcup\binom{[n-1]}{1}\sqcup\cdots\sqcup \binom{[n-1]}{d-1}$$ 
by $B_{i}$ for $i=1,2,\ldots,t=\sum\limits_{i=0}^{d-1}\binom{n-1}{i}$ such that $|B_{i}|\leq |B_{j}|$ for $i<j$. Let $w^{(i)}$ be
the characteristic vector of $B_{i}$ for each $i$. Let
\begin{equation*}
    I_{i}(x)=\prod\limits_{j\in B_{i}}x_{j}
\end{equation*}
for $i>1$, and $I_{1}(x)=1$. For $i=1,2,\ldots,t$, we define the polynomial $f_{i}$ as 
\begin{equation*}
    f_{i}(x)=(x_{n}-1)I_{i}(x).
\end{equation*}

Note that $f_i$ is multilinear and $f_i(w^{(i)}) \neq 0$ for each $i$, and $f_j(w^{(i)})=0$ for each $j>i$ since $B_j \not \subseteq B_i$. Thus, using the triangular criterion (see for example \cite[Proposition 2]{2007GC}),  $\{f_j\}_{j=1}^t$ are linearly independent over $\Q$.

Next we show that $\{p_i\}_{i=1}^r \cup \{f_j\}_{j=1}^t$ are linearly independent over $\Q$. Suppose otherwise that $\sum_{i=1}^r \alpha_i p_i +\sum_{j=1}^t \beta_j f_j=0$ for some coefficients that are not all zero; then not all $\alpha_i$ are zero, and not all $\beta_j$ are zero since we have shown that both families $\{p_i\}_{i=1}^r$ and $\{f_j\}_{j=1}^t$ are linearly independent over $\Q$. Note that for $i \leq r$, $n \notin A_i$ and thus 
$p_i$ does not depend on the variable $x_n$ in view of equation~\eqref{gi}. By setting $x_n=1$, we have $f_j(x)=0$ for each $1 \leq j \leq t$, which implies that $\sum_{i=1}^r \alpha_i p_i=0$, a contradiction. 

Finally we show that $\{p_i\}_{i=1}^{m} \cup \{f_j\}_{j=1}^t$ are linearly independent over $\Q$. Suppose otherwise that
\begin{equation}\label{=0}
\sum_{i=1}^m \alpha_i p_i +\sum_{j=1}^t \beta_j f_j=0    
\end{equation} with coefficients $\alpha_i, \beta_j$ being integers that are not all multiples of $p$. Then we must have $\alpha_k \neq 0$ for some $k>r$ since we have shown that $\{p_i\}_{i=1}^r \cup \{f_j\}_{j=1}^t$ are linearly independent. 

For each $k>r$ such that $\alpha_k \neq 0$, we have $n \in A_k$ and thus $v^{(k)}_n=1$. Therefore, by setting $x=v^{(k)}$ in equation~\eqref{=0}, we have $f_j(v^{(k)})=0$ for each $1 \leq j \leq t$ and thus
$$
\sum_{i=1}^m \alpha_i p_i(v^{(k)})+\sum_{j=1}^t \beta_j f_j(v^{(k)})=0 \implies \alpha_k p_k(v^{(k)})=-\sum_{i \neq k} \alpha_i p_i(v^{(k)}).
$$
Similar to the proof for the first part of the proposition, we must have $p \mid \alpha_k$.

For each $k\leq r$ such that $\alpha_k \neq 0$, we have $n \notin A_k$ and thus $v^{(k)}_n=0$. Recall that $p_k$ does not depend on the variable $x_n$. Let $u^{(k)}$ be the characteristic vector for $A_k \cup \{n\}$. We have $p_k(u^{(k)})=p_k(v^{(k)})=g(0)$ and $f_j(u^{(k)})=0$ for $1 \leq j \leq t$. 
By setting $x=u^{(k)}$ in equation~\eqref{=0}, we obtain that
$$
\alpha_k g(0)=\alpha_k p_k(u^{(k)})=-\sum_{i \neq k} \alpha_i p_i(u^{(k)})=-\sum_{i \neq k, i\leq r} \alpha_ig(|A_i \setminus A_k|)-\sum_{i>r} \alpha_ig(|A_i \setminus A_k|-1)
$$
(when $i\leq r$, note that $A_i \setminus (A_k \cup \{n\})=A_i \setminus A_k$ since $n \notin A_i$; when $i>r$, $|A_i \setminus (A_k \cup \{n\})|=|A_i \setminus A_k|-1$ since $n \in A_i$.)
For the right-hand side of the above equation, we have:
\begin{itemize}
    \item If $i<r$ and $i \neq k$, then $|A_i \setminus A_k| \in L \pmod q$ and thus $v_p(g(|A_i \setminus A_k|))>v_p(g(0))$.
    \item If $i>r$, then $i>k$ and thus $|A_i \setminus A_k| \in L \pmod q$. It follows that $v_p(g(|A_i \setminus A_k|-1))\geq v_p(g(0))$ by our assumption. Note that we have shown that $p \mid \alpha_i$ since $i>r$, so we still have $v_p(\alpha_ig(|A_i \setminus A_k|-1))=v_p(\alpha_i)+v_p(g(|A_i \setminus A_k|-1))\geq v_p(\alpha_i)+v_p(g(0))>v_p(g(0))$.
\end{itemize}
It follows from the ultrametric inequality that $v_p(\alpha_k g(0))>v_p(g(0))$, which implies that $p \mid \alpha_k$.

To conclude, we have deduced that $p \mid \alpha_i$ for all $1 \leq i \leq m$. Now
using the fact that $f_k(w^{(j)})=0$ for each $k>j$, by setting $x=w^{(j)}$ in equation~\eqref{=0} inductively on $j$, we can deduce that $p \mid \beta_j$ for each $1 \leq j \leq t$. Thus, all coefficients are multiples of $p$, contradicting our assumption. This establishes the linear independence of $\{p_i\}_{i=1}^m \cup \{f_j\}_{j=1}^t$. Note that these polynomials all lie in the space of multilinear polynomials in $n$ variables with degree at most $d$. By counting the dimension, we conclude that 
$$|\mathcal{F}|=m \leq \sum_{i=0}^d \binom{n}{i}-t=\sum_{i=0}^d \binom{n}{i}-\sum_{i=0}^{d-1} \binom{n-1}{i}=\sum_{i=0}^d \binom{n-1}{i}.
$$

Finally we briefly explain how to modify the proof if instead we have $v_p(g(0))\leq v_p(g(u+1))$ for each $u \in L+q\Z$. Let $\widetilde{f_j}=x_nI_{j}$ for each $1 \leq j \leq t$. It suffices to show $\{p_i\}_{i=1}^m \cup \{\widetilde{f_j}\}_{j=1}^{t}$ are linearly independent over $\Z$.  Suppose $\sum_{i=1}^m \alpha_i p_i +\sum_{j=1}^t \beta_j \widetilde{f_j}=0$ with coefficients not all multiples of $p$. Using similar arguments, for each $k \leq r$ with $\alpha_k \neq 0$, by setting $x=v^{(k)}$, we can show that $p \mid \alpha_k$; for each $k>r$ with $\alpha_k \neq 0$, by setting $x=\widetilde{u^{(k)}}$ (the characteristic vector for $A_k \setminus \{n\}$), we can show that $p \mid \alpha_k$. And finally, using the same argument, we can show $p \mid \beta_j$ for each $j$. To conclude the upper bound on $|\mathcal{F}|$, we use the same dimension counting argument.
\end{proof}

The readers are encouraged to jump to the proof of Theorem~\ref{[s] sym} at this point, where we use the same notations and refer to a few steps in the above proof, despite that Theorem~\ref{[s] sym} is about a non-modular version. For example, we will use the linear independence of $\{p_i\}_{i=1}^m \cup \{f_j\}_{j=1}^t$. The readers are also encouraged to glance at Section~\ref{sec:symmdiff} for a variant of Proposition~\ref{prop: p-adic}.

\subsection{Consequences of Proposition~\ref{prop: p-adic} and proof of Theorem~\ref{bchooses}}\label{4.2}

Throughout the section, we consider the following natural choice of the separating polynomial:
\begin{equation}\label{g}
g(y)=\prod_{\ell \in L}(y-\ell).   
\end{equation}
We will show that under extra assumptions on $L$, $g$ is indeed a separating polynomial and thus Proposition~\ref{prop: p-adic} can be applied to derive upper bounds on $q$-modular $L$-differencing Sperner systems.

As a quick application of Proposition~\ref{prop: p-adic}, we recover Theorem~\ref{LL}.
\begin{proof}[Proof of Theorem~\ref{LL}]
It suffices to show that the polynomial $g$ (defined in equation~\eqref{g}) satisfies the two assumptions in the statement of Proposition~\ref{prop: p-adic}. Note that $v_p(g(0))=0$ since no element in $L$ is a multiple of $p$. It follows that $v_p(g(u)) \geq 0=v_p(g(0))$ holds for each integer $u$. Moreover, if $u \in L+q\Z$, then there is some $\ell_0 \in L$ such that $q \mid (u-\ell_0)$ and thus $v_p(g(u)) \geq v_p(u-\ell_0) \geq v_p(q)>0=v_p(g(0))$.
\end{proof}

The following corollary can be regarded as a generalization of Theorem~\ref{LL}.

\begin{cor}\label{<k}
Let $q=p^k$ and let $L \subseteq [q-1]$. Let $\mathcal{F}\subseteq 2^{[n]}$ be a $q$-modular $L$-differencing Sperner system. If $\sum_{\ell \in L} v_p (\ell)<k$, then 
$$
  |\mathcal{F}| \leq \sum_{i=0}^{|L|} \binom{n}{i}.
$$
\end{cor}
\begin{proof}
It suffices to show that the polynomial $g$ satisfies the first assumption in the statement of Proposition~\ref{prop: p-adic}. Note that $v_p(g(0))=\sum_{\ell \in L} v_p(-\ell)=\sum_{\ell \in L} v_p(\ell)$. If $u \in L+q\Z$, then there is some $\ell_0 \in L$ such that $q \mid (u-\ell_0)$ and thus $v_p(g(u)) \geq v_p(u-\ell_0) \geq v_p(q)=k>v_p(g(0))$.
\end{proof}

\begin{rem}
Let $q=p^k$ and let $\mathcal{F}\subseteq 2^{[n]}$ be a $q$-modular $L$-differencing Sperner system. Note that if $L$ does not contain a multiple of $p^{k-1}$, then $\mathcal{F}\subseteq 2^{[n]}$ is a $p^{k-1}$-modular $L'$-differencing Sperner system, where $L'=\{ 1 \leq \ell <p^{k-1}: \ell \in L \pmod {p^{k-1}}\}$. Thus, the only non-degenerate case where the above corollary can be applied is that $L$ contains exactly one multiple of $p$, which is a multiple of $p^{k-1}$.
\end{rem}

Next, we use Proposition~\ref{prop: p-adic} to deduce Theorem~\ref{bchooses}, which gives a simple sufficient condition for the upper bound~\eqref{n-1i} to hold when $L$ is an interval.

\begin{proof}[Proof of Theorem~\ref{bchooses}]
It suffices to show that the polynomial $g$ satisfies the two assumptions in the statement of Proposition~\ref{prop: p-adic}. Note that $$v_p(g(0))=v_p(b(b-1) \cdots (b-s+1))=v_p\bigg(s!\binom{b}{s}\bigg)=v_p(s!)+v_p\bigg(\binom{b}{s}\bigg)=v_p(s!)$$ since $p \nmid \binom{b}{s}$. Thus, by Lemma~\ref{vps!}, $$v_p(g(u))=v_p((u-b+s-1)(u-b+s-2)\cdots (u-b)) \geq v_p(s!)=v_p(g(0))$$
for each integer $u$; moreover, if $u \in L+q\Z$, then $u \equiv \ell \pmod q$ for some $\ell \in L$ and thus we have $v_p(g(u))>v_p(s!)=v_p(g(0))$.
\end{proof}

Next, we use Theorem~\ref{bchooses} to show that there are many intervals $L$ for which upper bounds of the form~\eqref{n-1i} hold. 

\begin{thm}\label{lotsofL}
Given $n$ and a prime power $q=p^k$ with $k \geq 2$. There are at least $p^k(p-1)^k/2^k-q$ intervals $L \subseteq [q-1]$ such that the maximum size of $q$-modular $L$-differencing Sperner systems $\mathcal{F}\subseteq 2^{[n]}$ is at most $\sum_{i=0}^{|L|} \binom{n-1}{i}$.
\end{thm}
\begin{proof}
Let $N:=\#\{ (b,s): 1 \leq s\leq b<q, \quad p \nmid \binom{b}{s}\}$. Since each such pair $(b,s)$ gives rise to a desired interval $L=\{b-s+1, b-s+2,\ldots, b\}$, in view of Theorem~\ref{bchooses}, it then suffices to estimate the value $N$. Moreover, note that if $s>b$, then $\binom{b}{s}=0$ so that $p \mid \binom{b}{s}$; if $s=0$, then $\binom{b}{s}=1$ for any $0 \leq b<q$. Thus, $N=N'-q$, where 
$$
N'=\#\bigg\{ (b,s): 0 \leq s,b<q, \quad p \nmid \binom{b}{s}\bigg\}.
$$
For each $0 \leq s<q$, we write $s=(x_1x_2\ldots x_k)$ in its base-$p$ representation. By Corollary~\ref{nocarries}, the number of $b$ such that $0\leq b<q$ and $p \nmid \binom{b}{s}$ is $(p-x_1)(p-x_2)\cdots (p-x_k)$, which is the contribution of $s$ to $N'$. It follows that  
$$
N'=\sum_{\substack{0 \leq x_j \leq p-1\\ 1 \leq j \leq k}} \prod_{j=1}^{k} (p-x_j)=\prod_{j=1}^k \bigg(\sum_{x_j=0}^{p-1} (p-x_j)\bigg)=\bigg(\frac{p(p-1)}{2}\bigg)^k
$$
and $N=N'-q \approx q^2/2^k$.
\end{proof}

\begin{rem}
If $L \subseteq \{1,2,\ldots, q-1\}$ is an interval that does not contain a multiple of $p$ and 
$\mathcal{F}\subseteq 2^{[n]}$ such that $|A \setminus B| \in L \pmod q$ for distinct $A,B \in \mathcal{F}$. Then we can define 
$$
K=\{1\leq k \leq p-1: k \in L \pmod q\}
$$
such that $|K| \leq |L|$ and $|A \setminus B| \in K \pmod p$ for distinct $A,B \in \mathcal{F}$. In this case, we shall instead apply Theorem~\ref{LL}. However, the number of such intervals $L$ is bounded by $pq=o(q^2/2^k)$, so Theorem~\ref{lotsofL} still shows that there are many $L$ for which Theorem~\ref{LL} does not apply, and yet upper bounds of the form~\eqref{n-1i} still hold.
\end{rem}

Next, we prove Theorem~\ref{thm: AP}, which generalizes Theorem~\ref{bchooses} since it allows $L$ to be an arithmetic progression.

\begin{proof}[Proof of Theorem~\ref{thm: AP}]
It suffices to show that the polynomial $g$ satisfies the first assumption in the statement of Proposition~\ref{prop: p-adic}. Note that $v_p(g(0))=\sum_{\ell \in L} v_p(\ell)$. 

Let $u \in L+q\Z$; we need to show that $v_p(g(u))\geq \max\{(s-1)v_p(d)+v_p(q), sv_p(d)+ v_p(s!)+1\}$. Note that then $u \equiv \ell_0 \pmod q$ for some $\ell_0 \in L$ and thus $v_p(u-\ell)=v_p((u-\ell_0)+(\ell_0-\ell)) \geq v_p(\ell_0-\ell)\geq v_p(d)$ for each $\ell \in L$. It follows that $v_p(g(u))=\sum_{\ell \in L} v_p(u-\ell) \geq (s-1)v_p(d)+v_p(q)$. 

On the other hand, set $d=d'p^t$, where $t=v_p(d)$. Set $\ell_0=a+k_0d$. If $\ell=a+kd \in L$ such that $k \neq k_0$, then we have
\begin{align*}
v_p(u-\ell)
&=v_p((u-\ell_0)+(\ell_0-\ell))=v_p(\ell_0-\ell)\\
&=v_p((k-k_0)d)=t+v_p(k-k_0)=t+v_p((u-\ell_0)/p^t+k-k_0).
\end{align*}
Note that we also have $v_p(u-\ell_0)=t+v_p((u-\ell_0)/p^t)$.
It follows that
$$
v_p(g(u))=\sum_{\ell \in L} v_p(u-\ell)=st+\sum_{k=0}^{s-1} v_p\big((u-\ell_0)/p^t+k-k_0\big).
$$
Thus, by applying Lemma~\ref{vps!} (with $q/p^t$ being the modulus), we conclude that
$v_p(g(u))>sv_p(d)+v_p(s!)$.
\end{proof}

\begin{rem} 
If $L \subseteq [q-1]$ is contained in the arithmetic progression $a+d\Z$, one can follow the proof of the first part of Theorem~\ref{thm: AP} to see: if $\sum_{\ell \in L} v_p(\ell)<(|L|-1)v_p(d)+v_p(q)$, then $|\mathcal{F}|\leq \sum_{i=0}^{|L|}\binom{n}{i}$ for each $q$-modular $L$-differencing Sperner system $\mathcal{F}$ in $2^{[n]}$. 
\end{rem}

In the following example, we illustrate some special cases for which Theorem~\ref{thm: AP} can be applied:

\begin{ex}\label{ex: AP}
Let $L \subseteq [q-1]$ be an arithmetic progression $\{a,a+d,\ldots, a+(s-1)d\}$. In each of the following cases we have $\sum_{\ell \in L} v_p(\ell)<\max\{(s-1)v_p(d)+v_p(q), sv_p(d)+ v_p(s!)+1\}$ so that Theorem~\ref{thm: AP} applies to $q$-modular $L$-differencing Sperner systems.
\begin{itemize}
    \item $v_p(a)<v_p(d)$. In this case we have $v_p(a+kd)=v_p(a)$ for each $0 \leq k \leq s-1$ and thus $\sum_{\ell \in L} v_p(\ell)=sv_p(a)<sv_p(d)<(s-1)v_p(d)+v_p(q)$.
    \item $d \mid a$ and $p \nmid \binom{a/d+s-1}{s}$ (in particular, if $L$ is a homogeneous arithmetic progression of the form $\{d,2d, \ldots, sd\}$). In this case, we set $a=a'd$. Note that we have $v_p(a+kd)=v_p(d)+v_p(a'+k)$ for each $0 \leq k \leq s-1$. Thus, the condition $p \nmid \binom{a'+s-1}{s}$ implies that 
    $$
    \sum_{\ell \in L} v_p(\ell)=sv_p(d)+v_p\big(a'(a'+1)\cdots (a'+s-1)\big)=sv_p(d)+v_p(s!)<sv_p(d)+v_p(s!)+1.
    $$
\end{itemize}
\end{ex}

We finish the section by presenting the proof of Theorem~\ref{thm: closure}.

\begin{proof}[Proof of Theorem~\ref{thm: closure}]
The upper bound $|\mathcal{F}| \leq \sum_{i=0}^{2^{s-1}} \binom{n}{i}$ follows from Theorem~\ref{2^s}.

By Lemma~\ref{lem: closure}, we can find an interval $L' \subseteq [q-1]$ such that $L'$ is a $q$-closure of $L$ and $|L'|\leq \mu_q(s)$. Since $\mathcal{F}$ is $q$-modular $L$-differencing Sperner, it is also $q$-modular $L'$-differencing Sperner. Thus, we can apply Theorem~\ref{bchooses} to deduce that
$$
  |\mathcal{F}| \leq \sum_{i=0}^{|L'|} \binom{n-1}{i} \leq \sum_{i=0}^{\mu_q(s)} \binom{n-1}{i}.
$$

Next, we work on the case $q=p^2$. If $s+p-1 \leq 2s-1$, that is, $s \geq p$, then we are already done since $\mu_{p^2}(s) \leq s+p-1$. If $s<p$, then there is at most one element in $L$ being a multiple of $p$. Moreover, since $L \subseteq [q-1]$, we have $v_p(\ell)<v_p(q)$ for each $\ell \in L$. It follows that $\sum_{\ell \in L} v_p(\ell)<v_p(q)$ and the upper bound on $|\mathcal{F}|$ follows from Corollary~\ref{<k} .
\end{proof}

\begin{rem}
If $L$ is not an interval, then we can still apply Theorem~\ref{bchooses} by first finding a $q$-closure of $L$. Note that if $L$ is contained in an arithmetic progression, then one can instead first consider the ``closure" of $L$ with respect to that arithmetic progression and then apply Theorem~\ref{thm: AP}.

In this way, we can derive an upper bound on $|\mathcal{F}|$ (which would possibly depend on the arithmetic structure of $L$, in particular, the diameter of $L$), which provides a significant improvement on Theorem~\ref{2^s} if $|L|$ is much larger than $\log_2 q$, which is typically the case since $L \subseteq [q-1]$.
\end{rem}

\section{Proof of Theorem~\ref{[s] sym} and Theorem~\ref{thm: [s]}}\label{sec: push}
In this section, we combine the ``push-to-the-middle" idea and the linear algebra method to obtain new bounds on $L$-differencing Sperner systems and $L$-close Sperner systems. In the proofs, we will borrow some ideas from \cite{ABS91}.

We first prove Theorem~\ref{[s] sym}, which is a refined version of Theorem~\ref{bchooses} in the non-modular setting. 

\begin{proof}[Proof of Theorem~\ref{[s] sym}]
Let $\mathcal{F}=\{A_1,A_2, \ldots, A_m\} \subseteq 2^{[n]}$ be an $L$-differencing Sperner system. By Lemma~\ref{cor: middle}, we may assume that $s \leq |A_i| \leq n-s$ for each $1 \leq i \leq m$. Let $p>n$ be a prime. Note that $\mathcal{F}$ is also a $p$-modular $L$-differencing Sperner system with $L=\{1,2,\ldots, s\}$. Thus, from the proof of Theorem~\ref{bchooses}, the polynomial $g(y)=\prod_{\ell \in L}(y-\ell)$ satisfies the two assumptions in the statement of Proposition~\ref{prop: p-adic}. Thus, from the proof of Proposition~\ref{prop: p-adic}, we know that $\{p_i\}_{i=1}^m \cup \{f_j\}_{j=1}^t$ are linearly independent over $\Q$, where we follow all the notations in the proof of Proposition~\ref{prop: p-adic}. Note that these notations are independent of the choice of the prime $p$.

Let 
$$
Q(x)=\prod_{k=s-1}^{n-s} \bigg(\sum_{j=1}^{n-1} x_j -k\bigg).
$$
Label the sets in $$\binom{[n-1]}{0}\sqcup\binom{[n-1]}{1}\sqcup\cdots\sqcup \binom{[n-1]}{3s-n-2}$$ by $C_{i}$ for $i=1,2,\ldots,T=\sum\limits_{i=0}^{3s-n-2}\binom{n-1}{i}$ such that $|C_{i}|\leq |C_{j}|$ for $i<j$. For each $i$, let $z^{(i)}$ be
the characteristic vector of $C_{i}$ and define $h_i(x)$ to be the multilinear reduction of the polynomial
\begin{equation*}
    Q(x) \cdot \prod\limits_{j\in C_{i}}x_{j}
\end{equation*}
Note that each $h_i$ is a polynomial with degree at most $3s-n-2+(n-s)-s+2=s$. Moreover, note that for each $i$, $h_i(z^{(i)}) \neq 0$ since $|C_i| \leq 3s-n-2<s-1$, and $h_j(z^{(i)})=0$ for each $j>i$ since $C_j \not \subseteq C_i$. Thus, using the triangular criterion, $\{h_i\}_{i=1}^T$ are linearly independent over $\Q$. 

Next we show that $\{p_i\}_{i=1}^m \cup \{f_j\}_{j=1}^t \cup \{h_k\}_{k=1}^T$ are linearly independent over $\Q$. Suppose otherwise that
\begin{equation}\label{=000}
\sum_{i=1}^m \alpha_i p_i +\sum_{j=1}^t \beta_j f_j + \sum_{k=1}^T \gamma_k h_k=0    
\end{equation} 
for integer coefficients that are not all zero (again, the coefficients are independent of the choice of the prime $p$). Note that for each $1 \leq i \leq m$ and $1 \leq k \leq T$, we have $h_k(v^{(i)})=0$ since $s \leq |A_i| \leq n-s$ implies that $s-1 \leq \sum_{j=1}^{n-1} x_j \leq n-s$. Similarly, for each $1 \leq i \leq m$ and $1 \leq k \leq T$, we have $h_k(u^{(i)})=0$. It follows from equation~\eqref{=000} that
\begin{equation*}
\sum_{i=1}^m \alpha_i p_i(x) +\sum_{j=1}^t \beta_j f_j(x)=0    
\end{equation*} 
holds for $x=v^{(1)}, u^{(1)}, \ldots, v^{(m)}, u^{(m)}$. It follows from the proof of Proposition~\ref{prop: p-adic} that $p \mid \alpha_i$ for all $i$. By taking $p$ to be a sufficiently large prime, we must have $\alpha_i=0$ for all $i$. Therefore, equation~\eqref{=000} reduces to 
\begin{equation}\label{00}
0=\sum_{j=1}^t \beta_j f_j + \sum_{k=1}^T \gamma_k h_k= (x_n-1) \sum_{j=1}^t \beta_j I_j+ \sum_{k=1}^T \gamma_k h_k,
\end{equation} 
where $I_{j}$ is the same as in the proof of Proposition~\ref{prop: p-adic}. Note that $I_j$ and $h_k$ are independent of the variable $x_n$. Setting $x_n=1$ in equation~\eqref{00}, we obtain that $\sum_{k=1}^T \gamma_k h_k=0$; setting $x_n=0$ in equation~\eqref{00}, we obtain that $\sum_{j=1}^t \beta_j I_j=0$. Therefore, $\beta_j=0$ and $\gamma_k=0$ for all $j,k$, since we have shown that $\{I_j\}_{j=1}^t$ are linearly independent, and $\{h_k\}_{k=1}^T$ are linearly independent. 

We have established the linear independence of $\{p_i\}_{i=1}^m \cup \{f_j\}_{j=1}^t \cup \{h_k\}_{k=1}^T$. Note that these polynomials all lie in the space of multilinear polynomials in $n$ variables with degree at most $s$. By counting the dimension, we conclude that 
\[|\mathcal{F}|=m \leq \sum_{i=0}^s \binom{n-1}{i}-T= \sum_{i=0}^s \binom{n-1}{i}- \sum_{i=0}^{3s-n-2} \binom{n-1}{i}=\sum_{i=3s-n-1}^s \binom{n-1}{i}.\qedhere
\]
\end{proof}

Next, we use a similar strategy to prove Theorem~\ref{thm: [s]}.

\begin{proof}[Proof of Theorem~\ref{thm: [s]}]
Let $\mathcal{F}=\{A_1,A_2, \ldots, A_m\} \subseteq 2^{[n]}$ be an $L$-close Sperner system. By Lemma~\ref{cor: middle}, we may assume that $s \leq |A_i| \leq n-s$ for each $1 \leq i \leq m$. By relabeling, we may further assume that $|A_1| \geq |A_2| \geq \cdots \geq |A_m|$. For each $1 \leq i \leq m$, let $v^{(i)}$ be the characteristic vector of $A_i$ and define  $p_i$ to be the multilinear reduction of the polynomial
$$
\prod_{\ell \in L} \big(|A_i|- v^{(i)} \cdot x -\ell\big).
$$
Following \cite[Section 2]{2021GC}, a key observation is that $sd(F,G)=\min\{|F\setminus G|,|G\setminus F|\}=|F\setminus G|$ if and only if $|F|\leq |G|$. It follows that $p_i(v^{(i)}) \neq 0$ for each $i$ and $p_j(v^{(i)})=0$ for $j>i$. 

Let 
$$
Q(x)=\prod_{k=s}^{n-s} \bigg(\sum_{j=1}^{n} x_j -k\bigg).
$$
Label the sets in $$\binom{[n]}{0}\sqcup\binom{[n]}{1}\sqcup\cdots\sqcup \binom{[n]}{3s-n-1}$$ by $B_{i}$ for $i=1,2,\ldots,t=\sum\limits_{i=0}^{3s-n-1}\binom{n}{i}$ such that $|B_{i}|\leq |B_{j}|$ for $i<j$. For each $i$, let $w^{(i)}$ be
the characteristic vector of $B_{i}$ and define $f_i(x)$ to be the multilinear reduction of the polynomial
\begin{equation*}
    Q(x) \cdot \prod\limits_{j\in B_{i}}x_{j}
\end{equation*}
Note that each $f_i$ is a polynomial with degree at most $3s-n-1+(n-s)-s+1=s$. Moreover, note that for each $i$, $f_i(w^{(i)}) \neq 0$ since $|B_i| \leq 3s-n-1<s$, and $f_j(w^{(i)})=0$ for each $j>i$ since $B_j \not \subseteq B_i$. Thus, using the triangular criterion, $\{f_i\}_{i=1}^t$ are linearly independent over $\Q$. 

Next we show that $\{p_i\}_{i=1}^{m} \cup \{f_j\}_{j=1}^t$ are linearly independent over $\Q$. Suppose otherwise that
\begin{equation}\label{=00}
\sum_{i=1}^m \alpha_i p_i +\sum_{j=1}^t \beta_j f_j=0    
\end{equation} 
for some coefficients that are not all zero. Since we have shown that $\{f_j\}_{j=1}^t$ are linearly independent, not all $\alpha_i$ are zero. Let $k$ be the smallest integer such that $\alpha_k \neq 0$. Setting $x=v^{(k)}$ in equation~\eqref{=00}, we get $\alpha_k p_k(v^{(k)})+\sum_{j=1}^t \beta_j f_j(v^{(k)})=0$. Observe that $s \leq \sum_{j=1}^{n} x_j \leq n-s$ since $s \leq |A_k| \leq n-s$, and thus $f_j(v^{(k)})=0$. It follows that $\alpha_k=0$, which violates the assumption. 

Note that $\{p_i\}_{i=1}^{m} \cup \{f_j\}_{j=1}^t$ all lie in the space of multilinear polynomials in $n$ variables with degree at most $s$. By counting the dimension, we conclude that 
\[|\mathcal{F}|=m \leq \sum_{i=0}^s \binom{n}{i}-t=\sum_{i=0}^s \binom{n}{i}-\sum_{i=0}^{3s-n-1} \binom{n}{i}=\sum_{i=3s-n}^s \binom{n-1}{i}.\qedhere
\]
\end{proof}

\section{Applications to intersecting systems and set systems with restricted symmetric differences}\label{sec: Hamming}

\subsection{New upper bounds on $q$-modular $L$-avoiding $L$-intersecting systems}\label{sec:intersecting} 
In this subsection, we show how our arguments for Sperner systems can be modified to deduce improved upper bounds on intersecting systems. We remark that a weaker upper bound (with a cost of an extra multiplicative factor $(q-s)$) for each of the following results can be easily obtained by our main results on Sperner systems: we can first decompose a $q$-modular $L$-avoiding $L$-intersecting system into uniform subsystems (in the modulo $q$ sense) and realize that each uniform subsystem is a Sperner system with restricted differences (in the modulo $q$ sense). To remove the extra factor $(q-s)$, we need to return to the discussion on separating polynomials.

\begin{thm} \label{improve_FHR}
Let $L \subseteq \{0,1,\ldots, q-1\}$ be an interval (in the modulo $q$ sense) of size $s$ and let $\mathcal{F} \subseteq 2^{[n]}$ be a $q$-modular $L$-avoiding $L$-intersecting system. Then 
$$|\mathcal{F}|\leq \sum_{i=0}^{\mu_q(s)}\binom{n}{i}.$$
\end{thm}
\begin{proof}
Fix $k \not \in L \pmod q$. We define $L_k=\{0 \leq i \leq q-1: i \in k-L \pmod q\}$.  Since $k \notin L$, we have $0 \notin L_k$ and thus $L_k \subseteq [q-1]$. Moreover, since $L$ is an interval in the modulo $q$ sense, it follows that $L_k$ is an interval of size $s$ in $[q-1]$. By Lemma~\ref{lem: closure}, we can find an interval $L_k' \subseteq [q-1]$ such that $L_k'$ is a $q$-closure of $L_k$ with $|L_k'|\leq \mu_q(s)$. Thus, in view of the proof of Theorem~\ref{bchooses}, the polynomial $h_k(y):=\prod_{\ell \in L_k'} (y-\ell)$ separates $0$ from $L_k'+q\Z$. It follows that the polynomial $g_k(y):=h_k(k-y)$ separates $k$ from $L+q\Z$, and the degree of $g_k$ is at most $\mu_q(s)$. The upper bound on $|\mathcal{F}|$ follows immediately from Lemma~\ref{01}.
\end{proof}

Note that equation~\eqref{mu} implies that $\mu_q(s)<q-1$ whenever $s \neq q-1$. Thus, compared with Theorem~\ref{FHR}, Theorem~\ref{improve_FHR} provides a significant improvement if $s<q-1$. When $L=\{0,1,\ldots, s-1\}$ and $s\geq q/p$, we have $\mu_q(s) \leq s+q/p-1<2s$, and thus Theorem~\ref{improve_FHR} also improves Theorem~\ref{2s}.

\begin{rem}
Theorem~\ref{improve_FHR} is asymptotically tight when $\mu_q(s)=s$, which holds if $s=q-p^mt$ for some $1 \leq m \leq k-1$ and $1 \leq t \leq p$ in view of equation~\eqref{mu}. For example, we can consider the $(a+s)$-uniform system $\mathcal{F}=\{A \cup \{n-a+1, \ldots, n\}: A \in \binom{[n-a]}{s}\}$ for $L=\{a, \ldots, a+s-1\} \subseteq [q-1]$. Also note that when $s \leq p$, an analogue of Corollary~\ref{<k} holds and gives asymptotically tight upper bounds as well. It would be interesting to explore if our new upper bound is asymptotically tight when $s>p$ and $\mu_q(s)>s$.
\end{rem}

For a general $L$, we can combine the ideas used in the proofs of Theorem~\ref{improve_FHR} and Corollary~\ref{cor:qSperner} to prove the following result.

\begin{thm}\label{improve_2^sss}
Let $L \subseteq \{0,1,\ldots, q-1\}$ with $|L|=s$. Let $\mathcal{F} \subseteq 2^{[n]}$ be a $q$-modular $L$-avoiding $L$-intersecting system of sets. Then
\begin{equation*}
  |\mathcal{F}| \leq \sum_{i=0}^{q-1} \binom{n}{i}.
\end{equation*}
\end{thm}

Compared with Theorem~\ref{2^sss}, we see that Theorem~\ref{improve_2^sss} provides a significant improvement when $s >\log_2 q+1$.

In \cite[Question 2]{01JCTA}, Babai, Frankl, Kutin, and \v{S}tefankovi\v{c} asked the following question:

\begin{question}
In the case $q=p^2$, is it possible to improve the upper bound in Theorem~\ref{2^sss} from $O(n^{s^2/4+1})$ to $O(n^{cs})$ for some constant $c>0$?
\end{question}

The best lower bound, due to Kutin~\cite{K02}, has size $n^{s+\Omega(s^{1-\epsilon})}$, where $\epsilon>0$. Following the same idea used in the proof of Theorem~\ref{improve_FHR},  we give a positive answer to this question for all intervals $L$ by slightly modifying the proof of the second part of Theorem~\ref{thm: closure}.
\begin{thm}\label{p^2}
Let $p$ be a prime and let $q=p^2$. Let $L \subseteq \{0,1,\ldots, q-1\}$ be an interval (in the modulo $q$ sense) of size $s$ and let $\mathcal{F} \subseteq 2^{[n]}$ be a $q$-modular $L$-avoiding $L$-intersecting system. Then 
$$|\mathcal{F}|\leq \sum_{i=0}^{2s-1}\binom{n}{i}.$$
\end{thm}

\subsection{Set systems with restricted symmetric differences}\label{sec:symmdiff}

In this section, we explain how the analogues of our main results extend to the setting of set systems with restricted symmetric differences.

For a set $L$ of positive integers, following \cite{HKP20}, let $f_{L}(n)$ be the maximum size of subsets of the hypercube $\{0,1\}^n$ with pairwise Hamming distance in $L$. Equivalently, $f_L(n)$ is the maximum size of set systems $\mathcal{F} \subseteq 2^{[n]}$ such that $|A \triangle B| \in L$ for every distinct $A,B \in \mathcal{F}$. One can define the $q$-modular notion of $f_{L}(n)$ in a similar way: if $L \subseteq [q-1]$, we define $f_{L,q}(n)$ to be the maximum size of set systems $\mathcal{F} \subseteq 2^{[n]}$ such that $|A \triangle B| \in L \pmod q$ for every distinct $A,B \in \mathcal{F}$.

The quantity $f_L(n)$ has been studied extensively in the setting of coding theory and extremal combinatorics. Here we list a few related results. A celebrated theorem of Kleitman \cite{K66} determines $f_L(n)$ when $L=[s]$; in particular, $f_{[s]}(n)=\Theta(n^{\lfloor s/2 \rfloor})$. A classical result of Delsarte \cite{D73} shows that $f_{L}(n)\leq \sum_{i=0}^{|L|} \binom{n}{i}$ and Frankl \cite{F85} extended this to the $p$-modular version: $f_{L,p}(n)\leq \sum_{i=0}^{|L|} \binom{n}{i}$. Xu and Liu~\cite{XL12} showed that $f_{[s],q}(n)\leq \sum_{i=0}^{s} \binom{n}{i}$. In the setting of $\epsilon$-balanced codes, Alon \cite[Section 4]{A09} studied $f_L(n)$ when $L=[\frac{1-\epsilon}{2}\cdot n,\frac{1+\epsilon}{2}\cdot n] \cap \Z$. Note that if $L=[s]$ and $\mathcal{F} \subseteq 2^{[n]}$ is $L$-differencing Sperner, then we have the naive upper bound $|\mathcal{F}| \leq f_{[2s]}(n)$, which is much worse than the upper bounds shown in Theorem~\ref{LL} and Theorem~\ref{[s] sym}.

Recently, Huang, Klurman, and Pohoata \cite[Theorem 1.2]{HKP20} gave an algebraic proof of Kleitman's theorem and established the extension that $f_{L}(n)=O(n^{t-s})$ when $L=\{2s+1,\ldots,2t\}$ with $t>s\geq 0$. They also showed that $f_L(n)=O(n^c)$ if the number of even integers in $L$ is $c$, as $n \to \infty$ \cite[Theorem 3.5]{HKP20}. Moreover, they showed that if $q$ is a power of $2$ and $L=[q-1]$, then $f_{L,q}(n)=\Theta(n^{q/2-1})$ \cite[Theorem 3.6]{HKP20}.

Next, we explain why the same polynomial method also gives upper bounds on
\(f_{L,q}(n)\).  Indeed, the proof of the corresponding upper bounds on
\(f_{L,q}(n)\) only requires minimal changes.  The key observation is that the
following analogue of Proposition~\ref{prop: p-adic} holds.

\begin{prop}\label{prop:Hamming}
Let \(q=p^a\), and let \(L\subseteq [q-1]\).  If there exists a degree-\(d\)
univariate polynomial \(g\in\mathbb Z[x]\) separating \(0\) from \(L+q\mathbb Z\),
then
\begin{equation*}
        f_{L,q}(n) \leq \sum_{i=0}^{d} \binom{n}{i}.
\end{equation*}
If, in addition, the same polynomial \(g\) separates \(0\) from $(L-1)+q\mathbb Z$, where $ L-1:=\{\ell-1:\ell\in L\}$, 
then we have a stronger upper bound 
\begin{equation*}
        f_{L,q}(n) \leq \sum_{i=0}^{d} \binom{n-1}{i}.
\end{equation*}
\end{prop}

The proof of this analogue is essentially the same as the proof of
Proposition~\ref{prop: p-adic}.  For the first bound, one only needs to replace
the definition of \(g_i\) in equation~\eqref{gi} with
\begin{equation*}
        g_i(x)=g\bigl(|A_i|+\mathbf 1\cdot x-2v^{(i)}\cdot x\bigr).
\end{equation*}
Let \(p_i\) be the multilinear reduction of \(g_i\).  Then $p_i(v^{(j)})=g(|A_i\triangle A_j|)$
for each \(i,j\), and in particular \(p_i(v^{(i)})=g(0)\).  The same
\(p\)-adic linear-independence argument gives the first bound. For the stronger \((n-1)\)-variable bound, one applies the same argument after
deleting the \(n\)-th coordinate. The only new point is that deleting one
coordinate either preserves a Hamming distance or decreases it by one.  This is
why one must require separation not only from \(L+q\mathbb Z\), but also from
\((L-1)+q\mathbb Z\).

Using Proposition~\ref{prop:Hamming}, we can follow the arguments in
Section~\ref{sec: q-modular} to obtain Hamming-distance analogues of
Theorem~\ref{LL}, Theorem~\ref{bchooses}, Theorem~\ref{thm: AP}, and
Theorem~\ref{thm: closure}.  These analogues give bounds with
\(\sum_{i\le d}\binom ni\).  The stronger bounds with
\(\sum_{i\le d}\binom{n-1}{i}\) are available only when the same separating
polynomial also separates \(0\) from the shifted residue set
\((L-1)+q\mathbb Z\). For example, if \(L=\{2,3,\ldots,s+1\}\) and \(p>s+1\), then
\[
        f_{\{2,3,\ldots,s+1\},p}(n)
        \le
        \sum_{i=0}^{s+1}\binom{n-1}{i}.
\]

On the other hand, for intervals containing \(1\), such as \(L=[s]\), the above
\((n-1)\)-variable refinement does not apply directly, since \(0\in L-1\).  In
particular, one should not claim from Proposition~\ref{prop:Hamming} that $      f_{[s],q}(n)\le \sum_{i=0}^{s}\binom{n-1}{i}.$
This obstruction is already visible in the endpoint case \(L=[p-1]\).  Indeed,
after identifying a binary word with its support, the parameter \(T(n,p)\) studied by Bursics, Matolcsi, Pach, and Schrettner \cite{BMPS23} is exactly
\[
        T(n,p)=f_{[p-1],p}(n)
\]
Their result \cite[Theorem~8]{BMPS23} shows that 
\[
        f_{[p-1],p}(n)\le \sum_{i=0}^{p-1}\binom{n}{i},
\]
and moreover this bound is sharp when \(n\equiv -1\pmod p\).

Note that for certain subsets $L$, there are existing results that are better than the general upper bounds, for example, see~\cite[Theorem 3.4]{HKP20}. It will be interesting to explore if our techniques can be refined to obtain improved upper bounds on $f_{L,q}(n)$ for a larger class of subsets $L$ and prime powers $q$.

\section*{Acknowledgments}
The research of the first author is supported by the Institute for Basic Science (IBS-R029-C4). The second author thanks Gabriel Currier, Greg Martin, and Joshua Zahl for helpful discussions. The authors are also grateful to anonymous referees for their valuable comments and suggestions. The authors thank G\'abor Heged\H{u}s for pointing out an inaccuracy of the second part of Proposition~\ref{prop:Hamming} in the published version.

\bibliographystyle{abbrv}
\bibliography{CloseSperner}
\end{document}